\documentclass[a4paper,12pt]{article}
\usepackage{amsmath, amsthm, amstext}
\usepackage{amssymb, array, amsfonts}
\usepackage[utf8]{inputenc}
\usepackage[russian,english]{babel}
\usepackage{graphicx}
\usepackage[unicode]{hyperref}

\numberwithin{equation}{section}

\def \al{\alpha}
\def \be{\beta}
\def \ga{\gamma}
\def \de{\delta}
\def \er{\varepsilon}

\def \ka{\varkappa}
\def \la{\lambda}
\def \si{\sigma}

\def \oo{\omega}

\def \D{\Delta}

\def \N{\mathbb{N}}
\def \R{\mathbb{R}}

\def \T{\mathbb{T}}
\def\n{\nabla}
\def\dd{\partial}
\def\div{\operatorname{div}}

\def\1{1\!\!\!\!1}

\def\det{\operatorname{det}}

\def\im{\operatorname{Im}}
\def\re{\operatorname{Re}}
\DeclareMathOperator{\Texp}{T-exp}

\theoremstyle{plain}
\newtheorem{theorem}{\bf Theorem}[section]
\newtheorem{lemma}[theorem]{\bf Lemma}

\newtheorem{cor}[theorem]{\bf Corollary}

\theoremstyle{definition}
\newtheorem{defi}[theorem]{\bf Definition}

\theoremstyle{remark}
\newtheorem{rem}[theorem]{\bf Remark}

\renewcommand{\le}{\leqslant}
\renewcommand{\ge}{\geqslant}

\renewcommand{\qed}{\vrule height7pt width5pt depth0pt}

\title{Non-uniqueness~of~Leray-Hopf~solutions~for~a~dyadic~model}
\author{N.~Filonov, P.~Khodunov
\thanks{The work is supported by grant RFBR 17-01-00099-a.}}
\date{}
\begin{document}

\maketitle

\begin{abstract}
The dyadic model $\dot u_n + \lambda^{2n}u_n - \lambda^{\beta n}u_{n-1}^2 + \lambda^{\beta(n+1)}u_nu_{n+1} = f_n$,
$u_n(0)=0$, is considered. 
It is shown that in the case of non-trivial right hand side the system 
can have two different Leray-Hopf solutions.
\end{abstract}

\section*{Introduction}

Consider the following system of ordinary differential equations
\begin{equation}\label{system}
\left\{\begin{array}{l}
\dot u_n(t) + \lambda^{2n}u_n(t) - \lambda^{\beta n}u_{n-1}^2(t) + \lambda^{\beta(n+1)}u_n(t)u_{n+1}(t) = f_n(t),~~t\in[0;T],\\
u_n(0) = a_n,~~~n=1,2,\ldots.
\end{array}\right.
\end{equation}
Here $u_0 \equiv 0$; $\la > 1$, $\be > 0$ are parameters, $u_n, f_n$ are real valued functions.
We assume that initial data $\{a_n\} \in l_2$; 
and right-hand sides $f_n \in L_2(0,T)$, the behaviour of $f_n$ while $n \to \infty$ will be described later.

System \eqref{system} is similar to the system of the Navier-Stokes equations
\begin{equation}
\label{02}
\begin{cases}
\dd_t u - \D u + P \left((u, \n) u\right) = f 
\quad \text{in}\ \  [0, T] \times \T^d, \\
\div u = 0, \left.u\right|_{t=0} = a(x) .
\end{cases}
\end{equation}
Here $\T^d$ is a $d$-dimensional torus, $P$ is an orthogonal projector in $L_2(\T^d)$ on the subspace of solenoidal functions.
Both systems can be written in an abstract way
\begin{equation}
\label{03}
\begin{cases}
\dot u + A u + B (u, u) = f, \quad t \in [0, T] , \\
u (0) = a .
\end{cases}
\end{equation}
Function $u(t)$ here takes values in a Hilbert space ${\mathcal H}$, where
${\mathcal H}_{\eqref{system}} = l_2$ in the case \eqref{system}, and
$$
{\mathcal H}_{\eqref{02}} = \{u \in L_2 (\T^d, \R^d) : \div u = 0\}
$$ 
for the system \eqref{02}.
$A$ is a self-adjoint non-negative unbounded operator in ${\mathcal H}$,
$$
A_{\eqref{system}} \{u_n\} = \{\la^{2n} u_n\},
$$ 
and $A_{\eqref{02}} = - \D$.
Finally, $B$ is a bilinear unbounded map
$B : {\mathcal H} \times {\mathcal H} \to {\mathcal H}$,
$$
B_{\eqref{system}} (\{u_n\},\{v_n\}) = - \lambda^{\beta n}u_{n-1}(t)v_{n-1}(t) + \lambda^{\beta(n+1)}u_n(t)v_{n+1}(t),
$$
$$
B_{\eqref{02}}(u,v) = P((u, \nabla) v).
$$
The map $B$ has two important properties.

1) Ortogonality:
$$
\left(B(u,v), v\right)_{\mathcal H} = 0
$$
for a dense set of "good"\ $u$ and $v$.
For the system \eqref{system} we have
$$
\sum_{n=1}^\infty \left( - \la^{\be n} u_{n-1} v_{n-1} v_n 
+ \la^{\be (n+1)} u_n v_{n+1} v_n \right) = 0, 
$$
if all the series converge.
For \eqref{02} using the condition $\div u = 0$ one can get
$$
\int_{\T^d} \sum_{j,k=1}^3 u_j \dd_j v_k v_k\, dx 
= -\frac12 \int_{\T^d} \sum_{j=1}^3 \dd_j u_j |v|^2 \, dx = 0,
$$
if all the integrals converge.

2) Estimate
$$
\|B(u,u)\|_{\mathcal H} \le 
C \|A^{\si_1} u\|_{\mathcal H} \|A^{\si_2} u\|_{\mathcal H} .
$$
Exponents $\si_1$ and $\si_2$ can take any value, but their sum is fixed.
For the system \eqref{system} 
$$
\si_1 + \si_2 = \frac\be2.
$$
And for \eqref{02} we use the Cauchy inequality
$$
\int_{\T^d} \left|(u,\n)u\right|^2 dx \le
\left(\int_{\T^d} |u|^4 dx\right)^{1/2} \left(\int_{\T^d} |\n u|^4 dx\right)^{1/2}
$$
and embedding theorems
$$
W_2^{d/4} \subset L_4, \qquad W_2^{d/4 + 1} \subset W^1_4 .
$$
Thus, one could take 
$$
\si_1 = \frac{d}8, \quad \si_2 = \frac{d+4}8, 
\quad \text{and get} \quad \si_1 + \si_2 = \frac{d+2}4 .
$$

We consider the system \eqref{system} to be a model for Navier-Stokes equations. 
The space dimension $d=2$ in Navier-Stokes system corresponds to the value of the parameter $\be = 2$ in dyadic model \eqref{system}, 
and the dimension $d=3$ corresponds to the value $\be = 5/2$.
The explicit value of the parameter $\la > 1$ has no importance for us.

System \eqref{system} originates from the work \cite{DN} as a model of turbulence in hydrodynamics.

\begin{defi}
Suppose  $\{a_n\} \in l_2$, $f_n \in L_2(0,T)$ for all $n$, and 
$\sum_{n=1}^\infty \la^{-2n} \int_0^T f_n(t)^2 dt < \infty$.
A sequence of functions $\{u_n(t)\}_{n=1}^\infty$ is called a Leray-Hopf solution for system \eqref{system} if
\begin{itemize}
\item $u_n \in W_2^1 (0, T)$ and \eqref{system} hold for all $n$;
\item $\displaystyle\sup_{t\in [0,T]} \sum_{n=1}^\infty u_n(t)^2 < \infty$, 
$\displaystyle\sum_{n=1}^\infty \la^{2n} \int_0^T u_n(t)^2 dt < \infty$;
\item the estimate 
\begin{equation}
\label{leray}
\sum_{n=1}^\infty \left( u_n(t)^2 + 2 \la^{2n} \int_0^t u_n (\tau)^2 d\tau \right)
\le \sum_{n=1}^\infty \left(a_n^2 + 2 \int_0^t f_n (\tau) u_n (\tau)\, d\tau\right) 
\end{equation}
holds for all $t \in [0,T]$.
\end{itemize}
\end{defi}
It is easy to prove that such solutions always exist.

\begin{theorem}
\label{t01}
Suppose $\la>1$, $\be>0$,   $\{a_n\} \in l_2$,  
$\sum_{n=1}^\infty \la^{-2n} \int_0^T f_n(t)^2 dt < \infty$.
Then there esists a Leray-Hopf solution to the system \eqref{system}.
\end{theorem}

\begin{rem}
The condition $\sum_{n=1}^\infty \la^{-2n} \int_0^T f_n(t)^2 dt < \infty$
for system \eqref{system} 
is analogue of the condition $f \in L_2((0,T), W_2^{-1})$ for Navier-Stokes equations \eqref{02}.
\end{rem}

Theorem \ref{t01} is proved in \S \ref{s1}. 
In the previous work of the first author the uniqueness of Leray-Hopf solution was proved 
under the following assumptions.

\begin{theorem}[\cite{F}]
\label{t02}
Suppose $\la > 1$, $f_n(t) \equiv 0$.
Suppose $\{a_n\} \in l_2$ if $\be \le 2$, and
$a_n = o (\la^{(2-\be)n})$, $n \to \infty$, if $\be > 2$.
Then the Leray-Hopf solution to the system \eqref{system} is unique.
\end{theorem}

It is easy to prove that for $\be \le 2$ the Leray-Hopf solution is unique for a non-zero right-hand side as well.

\begin{theorem}
\label{t03}
Let $\la>1$, $\be\le 2$,   $\{a_n\} \in l_2$,  
$\sum_{n=1}^\infty \la^{-2n} \int_0^T f_n(t)^2 dt < \infty$.
Then the Leray-Hopf solution to the system \eqref{system} is unique.
\end{theorem}

This theorem is proved in \S \ref{s2}.

The question remains what happens if $\be > 2$ and the right-hand side is non-zero.
Let us formulate our main result.

\begin{theorem}
\label{t04}
Let $\la>1$, $\be>2$,   $a_n = 0$ for all $n$.
There exists $T>0$ and functions $\{ f_n(t)\}$ such that
$\sum_{n=1}^\infty \la^{-2n} \int_0^T f_n(t)^2 dt < \infty$, but system \eqref{system} has two different Leray-Hopf solutions.
\end{theorem}

\begin{rem}
As will be seen from the proof, the energy conservation holds for constructed solutions. 
So, for all $t\in [0, T]$ we have the equality in \eqref{leray}, see \eqref{38} below.
\end{rem}

\begin{rem}
The problem with non-zero but rapidly decreasing with $n$ right-hand sides is still open.
\end{rem}

System \eqref{system} and similar ones were considered also in works 
\cite{BMR, Ch, KP02, KZ, T, W}.

Let us note some results concerning strong solutions, 
although it is not directly related to our paper. 
The word "strong"\ here means fast decreasing  $u_n(t)$ with $n \to \infty$. 
Different authors give various definitions of strong solutions. 
In all cases the existence of a strong solution guarantees the uniqueness of the Leray-Hopf solution.

\begin{theorem}[\cite{Ch}]
\label{t05}
1) Suppose $\be \le 2$, $\sum_{n=1}^\infty \la^{2n} a_n^2 < \infty$,
$f_n (t) \ge 0$, $\sum_{n=1}^\infty \int_0^T f_n (t)^2 dt < \infty$.
Then there exists solution of \eqref{system}, such that estimate
$$
\sup_{t\in [0,T]} \sum_{n=1}^\infty \la^{2n} u_n (t)^2 < \infty 
$$
holds.

2) Let $\be > 3$, $\er > 0$, $f_n \equiv 0$.
Then there exists such a number $M(\er)$, that if
$a_n \ge 0$ and $\sum_{n=1}^\infty \la^{2\be\er n} a_n^2 > M$, then
for any solution $\{u_n\}$ of the system \eqref{system} 
$$
\int_0^T \left(\sum_{n=1}^\infty \la^{2(\er+1/3)\be n} 
u_n (t)^2\right)^{3/2} dt = + \infty
$$
holds for some finite $T$.
For example, one could take all $a_n = 0$ for $n \ge 2$ if $a_1$ is big enough.
\end{theorem}

The second part of this theorem in addition to the theorem \ref{t02} means that for $\be > 3$ the Leray-Hopf solutions can be not strong.

\begin{theorem}[\cite{BMR}]
\label{t06}
Let $\la = 2$, $2 < \be \le \frac52$.
Suppose $\{a_n\} \in l_2$, $a_n \ge 0$, $f_n \equiv 0$ for all $n$.
Then there exists a solution to \eqref{system}, such that
$$
u_n (t) = O (\la^{-\ga n}), \ \ n \to \infty, 
\quad \forall \ga, \ \ \forall t > 0 .
$$
\end{theorem}

The question of existence of the strong solution with arbitrary (not necessarily non-negative) "good"\ initial data and right-hand sides remains open. 
So does the question whether the Leray-Hopf solution is always strong for $2<\be\le 3$.

All three works \cite{BMR, Ch, F} are using the following property of positivity conservation in the absence of the right-hand sides 
(or with non-negative right-hand sides):

{\it if} 
$$
\dot u_n + \lambda^{2n}u_n  + \lambda^{\beta(n+1)}u_n u_{n+1} = \lambda^{\beta n}u_{n-1}^2,
$$
{\it and} $u_n(t_0) \ge 0$, {\it then} $u_n(t) \ge 0$ {\it for all} $t>t_0$.

This property follows from the explicit formula
\begin{eqnarray*}
u_n(t) = u_n(t_0) 
\exp\left(-\int_{t_0}^t (\la^{2n} + \la^{\be(n+1)} u_{n+1}(s)) ds\right) \\
+ \int_{t_0}^t \exp\left(-\int_s^t (\la^{2n} + \la^{\be(n+1)} u_{n+1}(\si)) d\si\right)
\la^{\be n} u_{n-1}(s)^2 ds .
\end{eqnarray*}
However, the conservation of positivity is a random property in a sense that firstly, Navier-Stokes equations do not have any analogous property, 
and secondly, that it is destroyed when a right-hand side is considered in the system \eqref{system}. 
This gave the authors the idea to build an example of non-uniqueness of Leray-Hopf solutions by choosing an appropriate right-hand side.

\subsection*{Idea of the proof}

The proof of the theorem \ref{t04} is based on an idea, originating from K.~Golovkin.
Now we get back to the abstract setting \eqref{03}.
Suppose that system \eqref{03} has two different solutions. We denote them as $u^\pm$ and rewrite them in a form
$$
u^\pm (t) = v(t) \pm g(t) ,
$$
where $v$ and $g$ are half-sum and half-difference of $u^{\pm}$ respectively. Then the system \eqref{03} is equivalent to:
\begin{equation}
\label{04}
\begin{cases}
\dot v + A v + B (v, v) + B (g, g) = f, \\
v(0) = a, \\
\dot g + A g + B (v, g) + B (g, v) = 0, \\
g(0) = 0 .
\end{cases}
\end{equation}

Note that the system on $g$ becomes {\it linear}.
Now we need to calibrate the coefficient $v$ in a way that the system on $g$ has a non-trivial solution. 
After that using the recently found $v$ and $g$ we calculate $f$, the right-hand side of the first equation in \eqref{04}, 
and make sure that it satisfies the requirements.

In our case the system \eqref{04} takes form
\begin{equation}
\label{new_system}
\left\{\begin{array}{l}
\dot v_n + \lambda^{2n}v_n - \lambda^{\beta n}v_{n-1}^2 - \lambda^{\beta n}g_{n-1}^2 + 
\lambda^{\beta(n+1)}v_n v_{n+1} + \lambda^{\beta(n+1)}g_n g_{n+1} = f_n,\\
v_n(0) = a_n,\\
\dot g_n + \lambda^{2n}g_n - 2\lambda^{\beta n} v_{n-1}g_{n-1} + \lambda^{\beta(n+1)}v_ng_{n+1} + \lambda^{\beta(n+1)}v_{n+1}g_n= 0,\\
g_n(0) = 0.
\end{array}\right.
\end{equation}

One can see that with $v_n \equiv 0$ the third equation of the system \eqref{new_system} becomes trivial. 
So we will split up $[0;T]$ into a set of intervals and put $v_n = 0$ on most of them. 
After that we need only to solve the problem on $g_n$ on a few intervals. 

Using scaling we are able to transform all the equations on $g_n$ to the unified form --- system of three equations on $[0;1]$ (see \eqref{31} below). 
To make $g_n$ continous one needs to add "gluing conditions"\ to the system. 
Existence of the solution of the system with such conditions is the subject of theorem \ref{t31}.

{\bf Plan of the paper.}
In \S \ref{s1} we prove that Leray-Hopf solution always exists.
In \S \ref{s2} we prove that Leray-Hopf solution is unique when $\be \le 2$.
In \S \ref{s3} we formulate theorem \ref{t31} and derive the main result (theorem \ref{t04}) from it.
In \S\S \ref{s4},\ref{s5} we prove the theorem \ref{t31}.

\section{Existence of a Leray-Hopf solution}
\label{s1}

For the sake of completeness we provide the proof of the existence of  Leray-Hopf solutions.

We introduce Galerkin solutions for the problem \eqref{system}. For any  $N \in \N$  consider the problem on the segment $[0;T]$
\begin{equation}
\label{11}
\begin{cases}
\dot v_n^{(N)} + \la^{2n} v_n^{(N)} - 
\la^{\be n} \left(v_{n-1}^{(N)}\right)^2 + 
\la^{\be(n+1)} v_n^{(N)} v_{n+1}^{(N)} = f_n, \quad n = 1, \dots, N,\\
v_n^{(N)} (0) = a_n, \qquad n = 1, \dots, N ; \qquad 
v_0^{(N)} \equiv v_{N+1}^{(N)} \equiv 0 .
\end{cases}
\end{equation}
It is equivalent to the system of integral equations
\begin{eqnarray}
\label{115}
v_n^{(N)} (t) = a_n + \int_0^t \biggl( f_n (\tau) - \la^{2n} v_n^{(N)} (\tau) + \la^{\be n} v_{n-1}^{(N)} (\tau)^2  \\
- \la^{\be(n+1)} v_n^{(N)} (\tau) v_{n+1}^{(N)} (\tau)\biggr) d\tau, \qquad n = 1, \dots, N ,
\nonumber
\end{eqnarray}
or one equation in $\R^N$ 
$$
v^{(N)} (t) = a_{(N)} + \int_0^t F_N (v^{(N)} (\tau), \tau) d\tau ,
$$
where $a_{(N)} = \begin{pmatrix} a_1\\ \dots\\ a_N\end{pmatrix}$, 
\begin{equation}
\label{12}
|F_N(y,\tau)| \le C_N \left(|y|+|y|^2\right) + |f(\tau)_{(N)}|
\end{equation}
and
\begin{equation}
\label{13}
|F_N(y,\tau) - F_N(z,\tau)| \le C_N (1+|y|+|z|) |y-z| .
\end{equation}
Denote
$$
R_N = 2 |a_{(N)}| + 2 \int_0^T |f(t)_{(N)}|\, dt, 
$$
\begin{equation}
\label{14}
\de_N = \frac1{2C_N (R_N + 1)} = 
\frac1{C_N \left(4 |a_{(N)}| + 4 \int_0^T |f(t)_{(N)}|\, dt +2\right)} .
\end{equation}
In the space of continuous functions $C([0,\de_N], \R^N)$ consider the closed ball
$$
M_N := \left\{ v \in C([0,\de_N], \R^N) : \|v\|_C \le R_N \right\}
$$
and the map
$$
(K_N v) (t) =  a_{(N)} + \int_0^t F_N (v (\tau), \tau) d\tau .
$$
It maps $M_N$ to itself due to \eqref{12} and \eqref{14}, and it is a contraction due to \eqref{13} and \eqref{14}. 
Thus, systems \eqref{115} and \eqref{11} have a solution on $[0,\de_N]$, where $\de_N$ is defined by \eqref{14}.
It is clear that $v_n^{(N)} \in W_2^1 (0, \de_N)$.
Multiplying $\eqref{11}$ with $2v_n^{(N)}$, summing for all $n$ and integrating we get
\begin{equation}
\label{15}
\sum_{n=1}^N v_n^{(N)} (t)^2 + 2 \sum_{n=1}^N \la^{2n} \int_0^t v_n^{(N)} (\tau)^2 d\tau 
= \sum_{n=1}^N a_n^2 + 2 \sum_{n=1}^N \int_0^t f_n (\tau) v_n^{(N)} (\tau) d\tau .
\end{equation}
Using Cauchi inequality for the last addend in the right-hand side, we arrive at the estimate
\begin{equation}
\label{16}
\sum_{n=1}^N v_n^{(N)} (t)^2 + \sum_{n=1}^N \la^{2n} \int_0^t v_n^{(N)} (\tau)^2 d\tau 
\le \sum_{n=1}^N a_n^2 + \sum_{n=1}^N \la^{-2n} \int_0^t f_n (\tau)^2 d\tau .
\end{equation}
So the following lemma is now proven.
\begin{lemma}
\label{l11}
System \eqref{11} has a solution on the segment $[0,t_1]$, 
$$
t_1 = \frac1{C_N \left(4 |v^{(N)}(0)| + 4 \int_0^T |f(t)_{(N)}|\, dt +2\right)} ,
$$
and
$$
|v^{(N)} (t_1)| \le |v^{(N)}(0)| + \left(\sum_{n=1}^N \la^{-2n} \int_0^T f_n(t)^2 dt\right)^{1/2} .
$$
\end{lemma}
After that we construct the solution on time intervals $[t_1,t_2]$, $[t_2,t_3]$ and so on.
And we have
$$
|v^{(N)} (t_k)| \le |a_{(N)}| + k \left(\sum_{n=1}^N \la^{-2n} \int_0^T f_n(t)^2 dt\right)^{1/2} ,
$$
$$
t_{k+1} - t_k \ge 
\frac1{C_N \left(4 |a_{(N)}| + 4 k \left(\sum_{n=1}^N \la^{-2n} \int_0^T f_n(t)^2 dt\right)^{1/2} 
+ 4 \int_0^T |f(t)_{(N)}|\, dt +2\right)}.
$$
Since the series $\sum_{k=1}^\infty (t_{k+1} - t_k)$ diverges, we get that system \eqref{11} has a solution on the entire interval $[0,T]$.
It also satisfies \eqref{16} for all $t\in [0,T]$, and hence
\begin{equation}
\label{17}
\sum_{n=1}^N v_n^{(N)} (t)^2 + \sum_{n=1}^N \la^{2n} \int_0^t v_n^{(N)} (\tau)^2 d\tau 
\le \sum_{n=1}^\infty a_n^2 + \sum_{n=1}^\infty \la^{-2n} \int_0^t f_n (\tau)^2 d\tau , \qquad \forall t \in [0,T].
\end{equation}
This inequality and the equation \eqref{115} imply that the sequence
$\{v_n^{(N)}\}_{N=n}^\infty$ is bounded in $W_2^1 (0,T)$ for any $n$.
Therefore there exists a sequence $\{v_n^{(N_k)}\}$, 
converging in $C[0,T]$ while $N_k \to \infty$.
Using a diagonal process we get the sequence of numbers $M_k$, such that
$$
v_n^{(M_k)} \underset{k\to\infty}\longrightarrow u_n \quad \text{in} \quad C[0,T] \qquad \forall n \in \N .
$$
We now show that the constructed sequence $\{u_n\}$ is a Leray-Hopf solution.
Indeed, substituting $N=M_k$ in \eqref{115} and going to the limit $k \to \infty$, 
one can get that the sequence $\{u_n (t)\}_{n=1}^\infty$ satisfies the system \eqref{system}.
Besides, $u_n \in W_2^1 (0,T)$ for all $n$.
Next, \eqref{17} yields
$$
\sum_{n=1}^N \left( v_n^{(M_k)} (t)^2 + \la^{2n} \int_0^t v_n^{(M_k)} (\tau)^2 d\tau \right)
\le \sum_{n=1}^\infty \left( a_n^2 + \la^{-2n} \int_0^T f_n (\tau)^2 d\tau \right) \qquad \text{with} \quad M_k \ge N.
$$
Taking a limit $k \to \infty$, we get 
$$
\sum_{n=1}^N \left( u_n (t)^2 + \la^{2n} \int_0^t u_n (\tau)^2 d\tau \right)
\le \sum_{n=1}^\infty \left( a_n^2 + \la^{-2n} \int_0^T f_n (\tau)^2 d\tau \right) .
$$
Due to the arbitrariness of $N$, this estimate guarantees that $\sum_{n=1}^\infty u_n (t)^2$ is bounded 
and that the series $\sum_{n=1}^\infty \la^{2n} \int_0^T u_n (\tau)^2 d\tau$ converges.

Now the only thing left to prove is the energy estimate.
We introduce the notation
$$
v^{(N)} = \{v_n^{(N)}\}_{n=1}^\infty, \quad \text{where} \quad v_n^{(N)} = 0 \ \ \text{with} \ \ n > N.
$$
It follows from \eqref{17} that the sequence $\{v^{(N)}\}_{N=1}^\infty$ is bounded in the Hilbert space of sequences of functions with a norm
$\left(\sum_{n=1}^\infty \la^{2n} \int_0^T v_n (t)^2 dt\right)^{1/2}$.
Without loss of generality, one can suppose that the sequence $\{v^{(M_k)}\}_{k=1}^\infty$ weakly converges in the mentioned space.
Also, all $v_n^{(M_k)}$ weakly converge $L_2(0,T)$, and therefore the limit coincides with the sequence $u = \{u_n(t)\}_{n=1}^\infty$.
Moreover, this weak convergence implies that
\begin{equation}
\label{18}
\sum_{n=1}^\infty \int_0^t f_n (\tau) v_n^{(M_k)} (\tau) \, d\tau 
 \underset{k\to\infty}\longrightarrow 
\sum_{n=1}^\infty \int_0^t f_n (\tau) u_n (\tau) \, d\tau \quad \forall t \in [0,T] .
\end{equation}
Finally, with $M_k \ge N$ we have
\begin{eqnarray*}
\sum_{n=1}^N \left(v_n^{(M_k)} (t)^2 + 2 \la^{2n} \int_0^t v_n^{(M_k)} (\tau)^2 d\tau\right) 
\le \sum_{n=1}^\infty \left(v_n^{(M_k)} (t)^2 + 2 \la^{2n} \int_0^t v_n^{(M_k)} (\tau)^2 d\tau\right) \\
= \sum_{n=1}^{M_k} a_n^2 + 2 \sum_{n=1}^\infty \int_0^t f_n (\tau) v_n^{(M_k)} (\tau)\, d\tau,
\end{eqnarray*}
where we used \eqref{15} in the second equality. Taking the limit $k\to\infty$ again, we get
$$
\sum_{n=1}^N \left(u_n (t)^2 + 2 \la^{2n} \int_0^t u_n (\tau)^2 d\tau\right) 
\le \sum_{n=1}^\infty \left(a_n^2 + 2 \int_0^t  f_n (\tau) u_n (\tau)\, d\tau\right) ,
$$
here we used the convergence \eqref{18}.
Because of the arbitrariness of $N$, the estimate \eqref{leray} follows.
Theorem \ref{t01} is proven.

\section{Uniqueness of the solution in the case $\be\le 2$}
\label{s2}

\begin{lemma}
\label{l21}
Let $u$ be a Leray-Hopf solution to the problem \eqref{system}, $\be \le 2$.
Then for any $\er > 0$ there exists $N$ such that
$$
\sum_{n=N}^\infty u_n(t)^2 \le \er \qquad \forall\ t \in [0,T].
$$
\end{lemma}

\begin{proof}
As $\be \le 2$, series 
$$
\sum_{n=1}^\infty \la^{\be n} \int_0^T |u_{n-1}(t)^2 u_n(t)| \, dt 
\le \sup_{n\in\N, t\in [0,T]} |u_n(t)| \cdot \sum_{n=1}^\infty \la^{2n} \int_0^T u_{n-1}(t)^2 dt < \infty
$$
converges by the definition of a Leray-Hopf solution.
So, multiplying \eqref{system} by $2u_n$,
integrating with respect to $t$ and summing up with respect to $n$ from $N$ to $\infty$, we get
\begin{eqnarray}
\label{21}
\sum_{n=N}^\infty \left(u_n (t)^2 + 2 \la^{2n} \int_0^t u_n (\tau)^2 d\tau\right) 
- 2 \la^{\be N} \int_0^t u_{N-1} (\tau)^2 u_N (\tau) \, d\tau \\
= \sum_{n=N}^\infty \left(a_n^2 + 2 \int_0^t  f_n (\tau) u_n (\tau)\, d\tau\right) .
\nonumber
\end{eqnarray}
It is clear that
$$
\sum_{n=N}^\infty a_n^2  \underset{N\to\infty}\longrightarrow 0,
$$
$$
\sum_{n=N}^\infty \left|  \int_0^t  f_n (\tau) u_n (\tau)\, d\tau\right| 
\le \left(\sum_{n=N}^\infty \la^{-2n} \int_0^T  f_n (\tau)^2 d\tau\right)^{1/2} 
\left(\sum_{n=N}^\infty \la^{2n} \int_0^T  u_n (\tau)^2 d\tau\right)^{1/2} \underset{N\to\infty}\longrightarrow 0,
$$
and
\begin{eqnarray*}
\la^{\be N} \left|\int_0^t u_{N-1} (\tau)^2 u_N (\tau) \, d\tau\right| 
\le C \la^{2N} \int_0^T  u_{N-1} (\tau)^2 d\tau 
\le C \la^2 \sum_{n=N-1}^\infty \la^{2n} \int_0^T  u_n (\tau)^2 d\tau  \underset{N\to\infty}\longrightarrow 0 .
\end{eqnarray*}
This convergence and \eqref{21} give the result.
\end{proof}

\begin{rem}
In particular, Lemma \ref{l21} implies that if $u$ is a Leray-Hopf solution to \eqref{system} and $\be \le 2$, then $u \in C\left([0,T]; l_2\right)$.
\end{rem}

\vspace{10pt}

{\it Proof of Theorem \ref{t03}.}
Suppose $u^\pm$ are two Leray-Hopf solutions of the system \eqref{system}.
We write them in the form $u^\pm = v \pm g$.
Then
\begin{equation}
\label{22}
\sup_{t\in (0,T)} \sum_{n=1}^\infty \left(v_n(t)^2 + g_n(t)^2\right) < \infty, \quad
\sum_{n=1}^\infty \la^{2n} \int_0^T \left(v_n(t)^2 + g_n(t)^2\right) dt < \infty ,
\end{equation}
Due to the last Lemma,
\begin{equation}
\label{23}
\forall \ \er \quad \exists \  N \  \colon \quad \sup_{t\in (0,T)} \sum_{n=N}^\infty v_n(t)^2 < \er .
\end{equation}
Functions $v_n$, $g_n$ satisfy the system \eqref{new_system}, therefore
$$
\dot g_n g_n + \lambda^{2n} g_n^2 - 2\lambda^{\beta n} v_{n-1} g_{n-1} g_n 
+ \lambda^{\beta(n+1)}v_n g_n g_{n+1} + \lambda^{\beta(n+1)}v_{n+1} g_n^2= 0.
$$
Integrating over $[0, t]$, summing up with respect to $n$ from $1$ to $\infty$ 
(all the series converge due to \eqref{22}) and using that $g_n(0) = 0$, we get
\begin{eqnarray*}
\sum_{n=1}^\infty \left(\frac{g_n (t)^2}2 + \la^{2n} \int_0^t g_n (\tau)^2 d\tau\right) 
= \sum_{n=1}^\infty \left( \la^{\be n} \int_0^t v_{n-1} g_{n-1} g_n d\tau 
- \la^{\be(n+1)} \int_0^t v_{n+1} g_n^2 d\tau\right) \\
\le C_1 \sum_{n=1}^\infty \la^{2n} \left(\|v_{n-1}\|_C \int_0^t (g_{n-1}^2 + g_n^2) d\tau 
+ \|v_{n+1}\|_C \int_0^t g_n^2 d\tau \right) \\
= C_1 \sum_{n=1}^\infty \la^{2n} \left(\|v_{n-1}\|_C + \|v_n\|_C + \|v_{n+1}\|_C\right)  \int_0^t g_n^2 d\tau .
\end{eqnarray*}
By virtue of \eqref{23} one can find such $N$ that
$$
\|v_{n-1}\|_C + \|v_n\|_C + \|v_{n+1}\|_C \le \frac1{C_1} \qquad \forall \ n > N .
$$
Then
$$
\sum_{n=1}^\infty g_n (t)^2  \le C_2 \sum_{n=1}^N \int_0^t g_n (\tau)^2 d\tau .
$$
Function $\displaystyle\Phi (t) = \sum_{n=1}^\infty g_n (t)^2$ is bounded due to \eqref{22}. 
Applying the Gronwall inequality to the function $\displaystyle\Phi$ we arrive at
$g_n(t) \equiv 0$ for all $n$.
Thus, $u^+ = u^-$. \qed

\section{Reduction to the system of three ODE}\label{s3}

We want to construct a non-trivial solution to the system \eqref{new_system} with initial data $a_n=0$.
The following theorem plays a key role in this construction.

\begin{theorem}
\label{t31}
Let $\la>1$, $\be>2$, $R>0$.
Consider on $[0;1]$ the system of ODE:
\begin{equation}
\label{31}
\begin{cases}
\dot h_1(\tau) = \left( -\lambda^{-2} + q(\tau)\right) h_1(\tau) - p(\tau)h_2(\tau),\\
\dot h_2(\tau) = 2p(\tau)h_1(\tau) - h_2(\tau) + \lambda^{\beta}q(\tau)h_3(\tau),\\
\dot h_3(\tau) = -  2\lambda^\beta q(\tau)h_2(\tau) - \la^2 h_3(\tau),\\
h_1(0) = 0, \quad h_2(0) = y, \quad h_3(0) = z.
\end{cases}
\end{equation}
There exist functions $p, q \in C_0^\infty (0, 1)$ and numbers $y,z\in\R$,
$y^2+z^2\neq 0$, such that the only solution $h \in C^\infty ([0,1]; \R^3)$ of the system \eqref{31} with given $p,q,y,z$ has the properties
\begin{equation}
\label{32}
h_1(1) = \rho y, \quad h_2(1) = \rho z, \qquad \text{where} \quad |\rho| > R .
\end{equation}
\end{theorem}

This theorem is proven in \S \ref{s5}.
Now we take $T=\frac1{\la^2-1}$
and divide the interval $(0,T)$ into an infinite set of subintervals.
Let $t_n = \frac1{(\la^2-1)\la^{2n}}$, then
$$
t_{n-1} - t_n = \la^{-2n}, \qquad (0,T) = \cup_{n=1}^\infty [t_n, t_{n-1}) .
$$
Suppose the functions  $p,q,h_1,h_2,h_3$ are given by the theorem \ref{t31} with sufficiently large $\rho$;
the value of $\rho$ will be chosen later.
Functions $v_n$ and $g_n$ will "start"\ not at $0$, but at the moment of time $t_{n+1}$.
Namely, we take
\begin{equation}
\label{vn}
v_n(t) = \left\{\begin{array}{lr}
0,& t < t_{n+1},\\
\lambda^{(2-\beta) (n+1)}p(\lambda^{2n+2}(t - t_{n+1})),&  t_{n+1} < t < t_n,\\
-\lambda^{(2-\beta) n}q(\lambda^{2n}(t - t_n)),&  t_n < t < t_{n-1},\\
0, & t > t_{n-1};
\end{array}\right.
\end{equation}

\begin{equation}
\label{gn}
g_n(t) = \left\{\begin{array}{lr}
0,& t < t_{n+1},\\
\rho^{-n-1}h_1(\lambda^{2n+2}(t - t_{n+1})),&  t_{n+1} < t < t_n,\\
\rho^{-n}h_2(\lambda^{2n}(t - t_n)),&  t_n < t < t_{n-1},\\
\rho^{-n+1}h_3(\lambda^{2n-2}(t - t_{n-1})),&  t_{n-1} < t < t_{n-2},\\
\rho^{-n+1}h_3(1) e^{-\la^{2n} (t-t_{n-2})},& t > t_{n-2}.
\end{array}\right.
\end{equation}
It is clear that $v_n \in C_0^\infty (0,T)$ and that the functions
$g_n$ are piecewise smooth, continuous at $t_{n+1}$, $t_n$, $t_{n-1}$, $t_{n-2}$ due to \eqref{31}, \eqref{32},
and therefore, $g_n \in W_2^1 (0,T)$.
It is also clear that 
$$
v_n(0) = g_n(0) = 0 \qquad \text{ for all} \  n,
$$
and that
\begin{equation}
\label{345}
v_n(t) = O (\la^{(2-\be)n}), \quad \dot v_n(t) = O (\la^{(4-\be)n}),
\quad g_n(t) = O (\rho^{-n}), \quad  n \to \infty .
\end{equation}

\begin{lemma}
\label{l32}
Suppose functions $v_n$ and $g_n$ are defined by the formulas \eqref{vn} and \eqref{gn} respectively.
Then
$$
\dot g_n + \lambda^{2n}g_n - 2\lambda^{\beta n} v_{n-1}g_{n-1} + \lambda^{\beta(n+1)}v_ng_{n+1} + \lambda^{\beta(n+1)}v_{n+1}g_n= 0 ,
$$
i.e. the third equation from \eqref{new_system} is satisfied.
\end{lemma}

\begin{proof}
Let us denote the left-hand side as $G_n(t)$.
One can see that $G_n(t) = 0$ while $t<t_{n+1}$.

While $t_{n+1} < t < t_n$ we introduce a new variable
$\tau = \lambda^{2n+2}(t - t_{n+1}) \in [0;1]$.
By definition $v_{n-1} \equiv 0$ on this time iterval, so
\begin{multline*}
G_n (t) = \lambda^{2n+2}\rho^{-n-1}\dot h_1(\tau) + \lambda^{2n}\rho^{-n-1}h_1(\tau) \\
+ \lambda^{\beta(n+1)}\lambda^{(2-\beta) (n+1)}p(\tau)\rho^{-n-1}h_2(\tau) 
- \lambda^{\beta(n+1)}\lambda^{(2-\beta)(n+1)}q(\tau)\rho^{-n-1}h_1(\tau)\\
= \la^{2n+2} \rho^{-n-1} \left(\dot h_1(\tau) + \lambda^{-2}h_1(\tau) + p(\tau)h_2(\tau) - q(\tau)h_1(\tau)\right)= 0
\end{multline*}
due to the first equation in \eqref{31}.

While $t_n < t < t_{n-1}$ we introduce a new variable
$\tau = \lambda^{2n}(t - t_n) \in [0;1]$.
On this time interval $v_{n+1} \equiv 0$, so
$$
G_n (t) = \la^{2n} \rho^{-n} \left(\dot h_2(\tau) + h_2(\tau) - 2p(\tau)h_1(\tau) - \lambda^{\beta}q(\tau)h_3(\tau)\right) = 0 
$$
due to the second equation in \eqref{31}.

While $t_{n-1} < t < t_{n-2}$ we introduce $\tau = \lambda^{2n-2}(t - t_{n-1}) \in [0;1]$;
here $v_n \equiv v_{n+1} \equiv 0$, and therefore
$$
G_n (t) = \la^{2n-2} \rho^{-n+1} \left(\dot h_3(\tau) + \la^2 h_3(\tau) + 2\lambda^\beta q(\tau)h_2(\tau)\right) = 0
$$
due to the third equation in \eqref{31}.

And finally, while $t > t_{n-2}$ 
\begin{multline*}
G_n (t) = \dot g_n (t) + \lambda^{2n}g_n(t) 
= \rho^{-n-1} h_3 (1) \left( -\la^{2n} e^{-\la^{2n} (t-t_{n-2})} +  \la^{2n} e^{-\la^{2n} (t-t_{n-2})}\right) = 0. \qquad
\qedhere
\end{multline*}
\end{proof}

\begin{lemma}
\label{l33}
We define functions $f_n$ from the first equation in \eqref{new_system}
with $v_n$ and $g_n$, given by formulas \eqref{vn} and \eqref{gn}.
If $|\rho| \ge \la^\be$, then
\begin{equation*}
f_n(t) = \left\{\begin{array}{lr}
0,& t < t_{n+1},\\
O (\lambda^{(4-\be)n}),&  t_{n+1} < t < t_{n-2},\\
O (\la^{-\be n}),& t > t_{n-2}.
\end{array}\right.
\end{equation*}
\end{lemma}

\begin{proof}
While $t<t_{n+1}$ we have $f_n(t)=0$ by construction.
The estimate $f_n(t)=O (\lambda^{(4-\be)n})$ follows from \eqref{345}.
The last estimate follows from \eqref{345} and the fact that
$$
f_n(t)= - \lambda^{\beta n}g_{n-1}(t)^2 + \lambda^{\beta(n+1)}g_n(t) g_{n+1}(t)
\qquad \text{with} \quad t > t_{n-2}. \quad \qedhere
$$
\end{proof}

\begin{cor}
\label{c34}
Under the conditions of Lemma \ref{l33}
$$
\sum_{n=1}^\infty \la^{-2n} \int_0^T f_n(t)^2 dt < \infty.
$$
\end{cor}

\begin{proof}
Using the previous lemma, we get
$$
\int_0^T f_n(t)^2 dt  = \int_{t_{n+1}}^{t_{n-2}} O (\la^{(8-2\be)n}) dt 
+ \int_{t_{n-2}}^T O (\la^{-2\be n}) dt = O (\la^{(6-2\be)n}) .
$$
Therefore the series $\sum_{n=1}^\infty \la^{-2n} \int_0^T f_n(t)^2 dt$ converges because $\be >2$.
\end{proof}

\begin{rem}
Moreover the series $\displaystyle\sum_{n=1}^\infty \la^{-\ga n} \int_0^T f_n(t)^2 dt$ converges for all $\ga>6-2\be$.
\end{rem}

\vspace{10pt}

{\it Proof of the theorem \ref{t04}.}
By the theorem \ref{t31} with $R = \la^\be$ we find functions $p$, $q$, $h_1$, $h_2$, $h_3$, satisfying \eqref{31} and \eqref{32}.
Using them, we construct functions $v_n$ and $g_n$ by formulae \eqref{vn} and \eqref{gn}.
Let $u_n^\pm (t) = v_n (t) \pm g_n (t)$, 
functions $f_n(t)$ we define from the first equation of \eqref{new_system}.
Then by Lemma \ref{l32} the system \eqref{new_system} is satisfied, and thus, \eqref{system}
is satisfied for functions $u_n^\pm$ with $a_n = 0$.
Furthermore, $\displaystyle\sum_{n=1}^\infty \la^{-2n} \int_0^T f_n(t)^2 dt < \infty$ due to the corollary \ref{c34}.
All functions $u_n^\pm \in W_2^1(0,T)$ and
\begin{equation}
\label{36}
u_n^\pm (t) = \left\{\begin{array}{lr}
0,& t < t_{n+1},\\
O (\lambda^{(2-\be)n}),&  t_{n+1} < t < t_{n-1},\\
O (\la^{-\be n}),& t > t_{n-1},
\end{array}\right.
\end{equation}
due to \eqref{345}.
Therefore,
$$
\sup_{t\in [0,T]} \sum_{n=1}^\infty u_n^\pm (t)^2 < \infty,
$$
and
$$
\int_0^T u_n^\pm(t)^2 dt = 
\int_{t_{n+1}}^{t_{n-1}} O (\la^{(4-2\be)n}) dt 
+ \int_{t_{n-1}}^T O (\la^{-2\be n}) dt = O (\la^{(2-2\be)n}) ,
$$
wherefrom
$$
\sum_{n=1}^\infty \la^{2n} \int_0^T u_n^\pm (t)^2 dt < \infty .
$$
We are left with the energy estimate.

Multiplying \eqref{system} by $u_n$, substituting $u_n = u_n^\pm$ and integrating over the time, we get
\begin{eqnarray}
\label{37}
\frac12 u_n^\pm (t)^2 + \la^{2n} \int_0^t u_n^\pm (\tau)^2 d\tau 
- \la^{\be n}  \int_0^t u_{n-1}^\pm (\tau)^2 u_n^\pm (\tau)\, d\tau \\
+ \la^{\be (n+1)}  \int_0^t u_n^\pm (\tau)^2 u_{n+1}^\pm (\tau)\, d\tau 
= \int_0^t f_n (\tau) u_n^\pm (\tau)\, d\tau .
\nonumber
\end{eqnarray}
By virtue of \eqref{36} 
$$
 \la^{\be n}  \int_0^T \left| u_{n-1}^\pm (\tau)^2 u_n^\pm (\tau)\right| \, d\tau 
= \la^{\be n} \int_{t_{n+1}}^{t_{n-2}} O (\la^{(6-3\be)n}) d\tau 
+ \int_{t_{n-2}}^T O (\la^{-2\be n}) d\tau = O (\la^{(4-2\be)n}) ,
$$
and hence we can sum up the equation \eqref{37} with respect to $n$ from $1$ to $\infty$, 
due to the absolute convergence of all the series.
We get
\begin{equation}
\label{38}
\sum_{n=1}^\infty \left( u_n^\pm(t)^2 + 2 \la^{2n} \int_0^t u_n^\pm (\tau)^2 d\tau \right)
= 2 \sum_{n=1}^\infty \int_0^t f_n (\tau) u_n^\pm (\tau)\, d\tau .
\end{equation}
So, $\{u_n^+ (t)\}_{n=1}^\infty$ and  $\{u_n^- (t)\}_{n=1}^\infty$ 
are Leray-Hopf solutions of the problem \eqref{system}.
They are distinct because $g_n \not\equiv 0$.
\qed

Now, Theorem \ref{t31} is the only thing left to prove. 

\section{The case of constant coefficients}
\label{s4}
In this section we consider the system \eqref{31} with constant coefficients
$p$ and $q$, moreover with $p = q/2$,
and prove an analog of Theorem \ref{t31} for this case.
In the next section we show that the statement of the theorem \ref{t31} is continuous 
with respect to changes of $p$ and $q$ in $L_1$-norm, and thus prove it for some functions $p,q \in C_0^\infty (0,1)$.

If $p$ and $q$ are constant, then the system \eqref{31} can be transformed into
\begin{equation}
\label{4000}
\begin{cases}
\dot h (\tau) = M h(\tau), \\
h(0) = v,
\end{cases}
\end{equation}
where
\begin{equation*}
h(\tau) = 
\begin{pmatrix} h_1 (\tau)\\ h_2 (\tau) \\ h_3 (\tau)\end{pmatrix} , \qquad 
v = \begin{pmatrix}0\\ y \\ z\end{pmatrix} \in \R^3 ,
\end{equation*}
\begin{equation}
\label{400}
M = \begin{pmatrix}
-\lambda^{-2} + q & -p & 0\\
2p & -1 & \lambda^\beta q\\
0 & -2\lambda^\beta q & -\lambda^2 
\end{pmatrix} .
\end{equation}
Then $h(1) = e^M v$. 
To satisfy \eqref{32}, we need to find values of $p$ and $q$ such that
there are sufficiently large numbers in the spectrum of matrices $M$ and $e^M$.

The following fact is well known, the proof can be found for example in 
\cite[Chapter II]{Kato}.

\begin{theorem}
\label{t41}
Suppose $T$ is a real $n\times n$ matrix with a simple spectrum,
$$
\si (T) = \{\la_1(T), \dots, \la_n(T)\}, \qquad \la_j(T) \neq \la_k(T)\ \ \text{for} \ \  j\neq k .
$$
Then for any $\er>0$ there exists $\de>0$, such that
if $\|T-S\| < \de$, then $|\la_j(S)-\la_j(T)| < \er$ with appropriate numeration of the spectrum $S$.
Moreover, if all the eigenvectors $v_j(T)$, $j=1,\dots,n$, have $v_j(T)_1 = 1$, 
then all the corresponding eigenvectors $v_j(S)$ of matrix $S$ can be chosen in such a way that $v_j(S)_1 = 1$ and $|v_j(T) - v_j(S)| < \er$.
\end{theorem}

We remind that in our case $\la>1$, $\be>2$.
Components of a three-dimensional vector $v_j$ will be denoted by $x_j$, $y_j$, $z_j$.

\begin{lemma}
\label{l42}
Consider the matrix
\begin{equation*}
A_0 = \begin{pmatrix}
1&-1/2&0\\
1&0&\lambda^\beta\\
0&-2\lambda^\beta &0
\end{pmatrix} .
\end{equation*}
The spectrum of $A_0$ is simple and
$$
\si (A_0) = \{ \ka^0, w^0, \bar w^0 \} ,
$$
where $\ka^0$ is real, moreover $\frac34<\ka^0<1$;
$w^0$ is not real, $0< \re w^0 < \frac18$, $\im w^0 > 0$. 
The corresponding eigenvectors can be chosen in a form
$$
v_1^0, \quad v_2^0 + i v_3^0, \quad v_2^0 - i v_3^0,
$$
with $v_1^0, v_2^0, v_3^0 \in \R^3$,
$$
x_1^0 = x_2^0 = 1, \quad x_3^0 = 0, \qquad y_1^0 > 0, \quad y_3^0 < 0, \quad z_3^0 > 0.
$$
\end{lemma}

\begin{proof}
The characteristic polynomial of the matrix $A_0$ has the form
\begin{equation*}
\chi (\alpha) = \alpha^3 - \alpha^2 + \left(\frac12 + 2\lambda^{2\beta}\right)\alpha - 2\lambda^{2\beta} .
\end{equation*}
Its derivative is positive everywhere on $\mathbb R$:
\begin{equation*}
\chi' (\alpha) = 3\alpha^2 - 2\alpha + \frac12 + 2\lambda^{2\beta} \ge \al^2 + 2\lambda^{2\beta} > 0 
\qquad \forall \ \al \in \R .
\end{equation*}
Hence, matrix $A_0$ has only one real eigenvalue and a complex conjugate pair of eigenvalues.
We denote them $\ka^0, w^0, \bar w^0$, with $\im w^0 > 0$.
By virtue of simple estimates
$$
\chi (1) = \frac12 > 0, \qquad 
\chi \left(\frac34\right) = \frac{15}{64} - \frac12\lambda^{2\beta}< 0 ,
$$
one has $\ka^0 \in (3/4, 1)$.
By Vieta's formulas, $\ka^0 + 2\re w^0 = 1$, therefore $\re w^0 \in (0, 1/8)$.

One can easily see that the first components of eigenvectors can not be zero, 
so they can be chosen unitary, i.e. $x_1^0 = x_2^0 = 1$, $x_3^0 = 0$.
Next, it follows from $A_0 v_1^0 = \ka^0 v_1^0$ that 
$$
y_1^0 = 2 (1-\ka^0) > 0.
$$
Finally, the equality $A_0 (v_2^0 + i v_3^0) = w^0 (v_2^0 + i v_3^0)$ implies
$$
y_3^0 = - 2 \im w_0 < 0 \quad \text{and} \quad 
z_3^0 = \im \frac{4 \la^\be (w^0-1)}{w^0} = \frac{4 \la^\be \im w^0}{|w^0|^2} > 0 .
\quad \qedhere
$$
\end{proof}

\begin{cor}
\label{c43}
There exist numbers $q_0$, $\mu>1$ and $\nu \in (0,1)$, such that for $q>q_0$, the matrix
\begin{equation*}
A = 
\begin{pmatrix}
1 - \frac{\lambda^{-2}}q & -\frac12 & 0\\
1 & -\frac1q & \lambda^\beta \\
0 & -2\lambda^\beta  & -\frac{\lambda^2}q
\end{pmatrix}
\end{equation*}
has a simple spectrum
$$
\si (A) = \{ \ka, w, \bar w \} , \qquad 
\frac34<\ka<1, \quad 0< \re w < \frac18, \quad \im w > 0 ,
$$
and the corresponding eigenvectors $v_1$, $v_2 + i v_3$, $v_2 - i v_3$ can be chosen in such a way that
\begin{equation}
\label{40}
x_1 = x_2 = 1, \quad x_3 = 0, 
\end{equation}
$$
y_1 > \frac{y_1^0}2, \quad y_3 < \frac{y_3^0}2, \quad z_3 > \frac{z_3^0}2 ,
$$
\begin{equation}
\label{405}
|v_1| + |v_2| + |v_3| \le \mu, \qquad z_3 - y_1 y_3 \ge \nu .
\end{equation}
\end{cor}

\begin{proof}
Note that $\|A - A_0\| = O(\frac1q)$. 
Therefore for a sufficiently large $q_0$ all the inequalities follow from Lemma \ref{l42} and Theorem \ref{t41}.
For $\mu$ and $\nu$ one can take
$$
\mu = 2\left(|v_1^0| + |v_2^0| + |v_3^0|\right), \qquad 
\nu = \min\left(\frac{z_3^0}2 - \frac{y_1^0y_3^0}4; \frac12\right).
\quad \qedhere
$$
\end{proof}

\begin{lemma}
\label{l44}
Let $A$, $\mu$, $\nu$ be the matrix and the numbers from Corollary \ref{c43}.
Let $R>0$.
There exists a number $q_1$ such that for $q>q_1$ the matrix 
$$
B = e^{qA} = \exp 
\begin{pmatrix}
q -\lambda^{-2} & -q/2 & 0\\
q & -1 & \lambda^\beta q\\
0 & -2\lambda^\beta q & -\lambda^2
\end{pmatrix}
$$
has the same eigenvectors $v_1$, $v_2 + i v_3$, $v_2 - i v_3$,
as the matrix $A$, and the eigenvalues
$$
e^{q\ka} =: k, \quad  e^{qw} =: a + ib, \quad e^{q\bar w} = a - ib,
$$
\begin{equation}
\label{41}
B v_1 = kv_1, \quad B v_2 = av_2 - bv_3, \quad B v_3 = bv_2 + av_3 .
\end{equation}
Moreover,
\begin{equation}
\label{415}
k > \frac{5 \mu^3 R}{2\nu},
\end{equation}
\begin{equation}
\label{42}
|a| < \oo k, \qquad |b| < \oo k,
\end{equation}
where
\begin{equation}
\label{43}
\oo < \frac{\nu^2}{100 \mu^4} < \frac{\nu}{100\mu^2} < \frac1{100} .
\end{equation}
\end{lemma}

\begin{proof}
As $\ka > 3/4$, one has $k > e^{3q/4}$.
Next, $\re w < 1/8$ yields 
$$
\max (|a|,|b|) \le |e^{qw}| = e^{q\re w} \le e^{q/8} \le e^{-5q/8} k .
$$
Therefore $\oo$ can be chosen to be $\oo = e^{-5q_1/8}$.
If one chooses $q_1$ so big that 
$$
e^{3q_1/4} > \frac{5 \mu^3 R}{2\nu} \quad \text{and} \quad 
e^{-5q_1/8} < \frac{\nu^2}{100 \mu^4},
$$
then the conditions \eqref{415}, \eqref{42} and \eqref{43} will be fulfilled.
\end{proof}

We fix $2p=q>q_1$, where $q_1$ is a number from Lemma \ref{l44}.
Suppose $h$ is a solution to the system \eqref{4000}, \eqref{400} with such $p$ and $q$.
Then the condition \eqref{32} is equivalent to the system
\begin{equation}
\label{44}
\left(B\begin{pmatrix}
0\\y\\z
\end{pmatrix}\right)_1 = \rho y, \qquad
\left(B\begin{pmatrix}
0\\y\\z
\end{pmatrix}\right)_2 = \rho z ,
 \end{equation}
where $B = e^M$ is a matrix from Lemma \ref{l44}.
We denote
\begin{equation}
\label{45}
\tilde B = 
\begin{pmatrix}
b_{12} & b_{13}\\
b_{22} & b_{23}
\end{pmatrix} .
\end{equation}
One can see that the existence of a non-trivial solution $\begin{pmatrix}0\\y\\z\end{pmatrix}$ of a system \eqref{44} 
is equivalent to $\rho \in \si (\tilde B)$.

\begin{lemma}
\label{l45}
Let $R>0$. 
Suppose $q_1$ and $B$ are defined by $R$ as in Lemma \ref{l44}.
Then for $q>q_1$ the matrix $\tilde B$ has two different eigenvalues $\rho_1$, $\rho_2$, and $\max(|\rho_1|, |\rho_2|) > R$.
\end{lemma}

\begin{proof}
We will search for a solution of \eqref{44} in a form
$$
v = \begin{pmatrix}0\\y\\z\end{pmatrix} = c_1v_1 + c_2v_2 + c_3v_3, 
$$
where $v_1, v_2,v_3$ satisfy \eqref{41}, $c_1,c_2,c_3 \in \R$.
Then we get the following system 
\begin{equation*}
\left\{\begin{array}{l}
c_1 x_1 + c_2 x_2 + c_3 x_3 = 0,\\
c_1(\rho y_1 - kx_1) + c_2(\rho y_2 - ax_2 + bx_3) + c_3(\rho y_3 - bx_2 - ax_3) = 0,\\
c_1(\rho z_1 - ky_1) + c_2(\rho z_2 - ay_2 + by_3) + c_3(\rho z_3 - by_2 - ay_3) = 0,
\end{array}\right.
\end{equation*}
of three linear equations on $c_i$.
Its determinant is equal to
\begin{multline*}
\mathrm{det} = x_1\left((\rho y_2 - ax_2 + bx_3)(\rho z_3 - by_2 - ay_3) - (\rho z_2 - ay_2 + by_3)(\rho y_3 - bx_2 - ax_3)\right)  \\ - 
x_2\left((\rho y_1 - kx_1)(\rho z_3 - by_2 - ay_3) - (\rho z_1 - ky_1)(\rho y_3 - bx_2 - ax_3)\right)  \\ + 
x_3\left((\rho y_1 - kx_1)(\rho z_2 - ay_2 + by_3) - (\rho z_1 - ky_1)(\rho y_2 - ax_2 + bx_3)\right) \\
= U \rho^2 + V \rho + W, 
\end{multline*}
where
\begin{eqnarray*}
U = \det (v_1, v_2, v_3) \neq 0 ,\\
V = (k-a) (z_3 - y_1y_3) + b (y_1y_2 - y_2^2 - y_3^2 + z_2 - z_1) ,\\
W = (a^2+b^2) y_3 - k (ay_3 + by_2 - by_1) ;
\end{eqnarray*}
in the last two equalities we used the relations \eqref{40}.
Using \eqref{405}, \eqref{42} and \eqref{43} we get
\begin{eqnarray*}
|U| \le |v_1| | v_2| |v_3| \le \mu^3 , \\
V \ge \frac{9k\nu}{10} - 5 |b| \mu^2 \ge \frac{4k\nu}5 , \\
|W| \le \mu (a^2 + b^2 + k|a| + 2k|b|) 
\le \mu (2\oo^2 + 3\oo) k^2 \le \frac{\nu^2k^2}{25\mu^3} .
\end{eqnarray*}
Therefore
$$
V^2 - 4 U W \ge \frac{16k^2\nu^2}{25} - \frac{4k^2\nu^2}{25} > 0,
$$
which means that the discriminant is positive and the equation
$$
U \rho^2 + V \rho + W = 0
$$
has two distinct roots.
By Vieta's formulas, one can see that
$$
|\rho_1 + \rho_2| = \frac{V}{|U|} \ge \frac{4k\nu}{5\mu^3} > 2 R
$$
due to \eqref{415}. 
Thus, there is a root $\rho$ such that $|\rho| > R$.
\end{proof}

\section{Proof of Theorem \ref{t31}}
\label{s5}
Let $A \in L_1 \left((a,b); \mathrm{Mat} (n\times n, \R)\right)$.
Consider on the interval $(a,b)$ the following Cauchy problem
\begin{equation*}
\begin{cases}
\dot h (t) = A(t) h(t), \\
h(a) = h_0;
\end{cases}
\end{equation*}
here $h(t) \in\R^n$.
It is equivalent to the integral equation
\begin{equation}
\label{51}
h(t) = h_0 + \int_a^t A(\tau) h(\tau) \, d\tau .
\end{equation}
It is well known (see for example \cite[Chapter III, \S 31]{MSch}), that there exists a unique solution
$h \in C\left([a,b]; \R^n\right)$
(and therefore $h \in W_1^1\left((a,b); \R^n\right)$).
Iterating \eqref{51}, we get 
\begin{equation}
\label{52}
h(t) =  \sum_{m=0}^\infty \frac1{m!}\int\limits_{(a;t)^m} 
\mathcal T\left(A(\tau_1),\ldots, A(\tau_m)\right) \, d\tau_1\ldots d\tau_m\, h_0 ,
\end{equation}
where the symbol $\mathcal T$ denotes a chronological ordering
$$
\mathcal T\left(A(\tau_1),\ldots, A(\tau_m)\right) = A(\si_1) \ldots A(\si_m) ,
$$
$\si_1 \ge \si_2 \ge \dots \ge \si_m$ is a non-increasing permutation of the arguments $\tau_1, \dots, \tau_m$.
From \eqref{52} follows a well known (can be found in \cite{MSch}) estimate
$$
\|h\|_C \le e^{\|A\|_{L_1}} |h_0| .
$$
We denote by $B$ a linear operator mapping initial data $h_0$ to the final value $h(b)$. 
Matrix $B \in \mathrm{Mat} (n\times n, \R)$ is a $T$-exponent of the matrix function $A$:
$$
B = \Texp(A) \equiv  \sum_{m=0}^\infty \frac1{m!}\int\limits_{(a;b)^m} 
\mathcal T\left(A(\tau_1),\ldots, A(\tau_m)\right) \, d\tau_1\ldots d\tau_m .
$$
The map
$$
\Texp : L_1 \left((a,b); \mathrm{Mat} (n\times n)\right) \to  \mathrm{Mat} (n\times n) 
$$
is continuous.

\begin{theorem}
\label{t51}
Let $A_1, A_2 \in  L_1 \left((a,b); \mathrm{Mat} (n\times n)\right)$,
$B_1 = \Texp(A_1)$, $B_2 = \Texp(A_2)$.
Then
$$
\|B_1 - B_2\| \le \exp\left(\max(\|A_1\|_{L_1}, \|A_2\|_{L_1})\right) 
 \int\limits_a^b \| A_1(\tau) - A_2(\tau)\|\, d\tau .
$$
\end{theorem}

However, we could not find the reference in the literature. 
For the convenience of the reader the proof of the theorem is given at the end of this section.

\vspace{7pt}

Let us come back to the proof of the main result.

{\it Proof of Theorem \ref{t31}.}
Lemmas \ref{l44} and \ref{l45} provide us with numbers $q_*$ and $\rho_*$, $|\rho_*| > R$, 
and a non-zero vector $\begin{pmatrix}y_*\\z_*\end{pmatrix} \in \R^2$, such that
$$
\tilde B_* \begin{pmatrix}y_*\\z_*\end{pmatrix} = \rho_* \begin{pmatrix}y_*\\z_*\end{pmatrix} ,
$$
where matrix $\tilde B_*$ is defined by the formula \eqref{45} using the matrix
$$
B_* = e^{M_*}, \qquad M_* = \begin{pmatrix}
q_* -\lambda^{-2} & -q_*/2 & 0\\
q_* & -1 & \lambda^\beta q_*\\
0 & -2\lambda^\beta q_* & -\lambda^2 
\end{pmatrix} .
$$
Note that the eigenvalues of $\tilde B_*$ are distinct.

Now we find functions $p, q \in C_0^\infty (0,1)$, that are close to constants $q_*/2$ and $q_*$ in the sense of $L_1 (0,1)$-norm.
We construct the corresponding matrix function $M(t)$ using \eqref{400}
and matrix $B = \Texp(M)$. 
By Theorem \ref{t51} the norm of the difference $\|B-B_*\|$ (and therefore the norm $\|\tilde B - \tilde B_*\|$) 
can be made arbitrarily small by taking the functions $p$ and $q$ close to the numbers $q_*/2$ and $q_*$.
Now Theorem \ref{t41} guarantees that the spectrum of the matrix $\tilde B$ consists of two real numbers, 
and one them is such $\rho$ that $|\rho| > R$, 
$$
\tilde B \begin{pmatrix}y\\z\end{pmatrix} = \rho \begin{pmatrix}y\\z\end{pmatrix} , \qquad y^2 + z^2 > 0 .
$$
Therefore, the solution to the system \eqref{31} with such $p,q,y,z$ satisfies the condition \eqref{32}.
\qed

\subsection{Continuity of $T$-exponent}

\begin{proof}[Proof of Theorem \ref{t51}]
We have
\begin{eqnarray*}
A_1 (\tau_1) \dots A_1 (\tau_m) - A_2 (\tau_1) \dots A_2 (\tau_m)  \\
= \sum_{k=1}^m A_1 (\tau_1) \dots A_1 (\tau_{k-1}) \left(A_1(\tau_k) - A_2(\tau_k)\right)
A_2 (\tau_{k+1})  \dots A_2 (\tau_m) ,
\end{eqnarray*}
wherefrom
\begin{multline*}
\left\| \int\limits_{(a;b)^m} \left(\mathcal T\left(A_1(\tau_1),\ldots, A_1(\tau_m)\right) 
- \mathcal T\left(A_2(\tau_1),\ldots, A_2(\tau_m)\right) \right) d\tau_1\ldots d\tau_m \right\| \\
\le \sum_{k=1}^m \int\limits_{(a;b)^m} \| A_1 (\tau_1)\| \dots \|A_1 (\tau_{k-1})\| \left\|A_1(\tau_k) - A_2(\tau_k)\right\|
\|A_2 (\tau_{k+1})\|  \dots \|A_2 (\tau_m) \| d\tau_1\ldots d\tau_m \\
= \sum_{k=1}^m \|A_1\|_{L_1}^{k-1} \|A_1-A_2\|_{L_1} \|A_2\|_{L_1}^{m-k} 
\le m L^{m-1} \|A_1-A_2\|_{L_1} ,
\end{multline*}
with $L = \max(\|A_1\|_{L_1}, \|A_2\|_{L_1})$.
Therefore
\begin{multline*}
\|B_1 - B_2\| \le 
 \sum_{m=1}^\infty \frac1{m!} \left\| \int\limits_{(a;b)^m} \left(\mathcal T\left(A_1(\tau_1),\ldots, A_1(\tau_m)\right) 
- \mathcal T\left(A_2(\tau_1),\ldots, A_2(\tau_m)\right)\right) d\tau_1\ldots d\tau_m \right\| \\
\le \sum_{m=1}^\infty \frac{L^{m-1}}{(m-1)!} \|A_1-A_2\|_{L_1}
= e^L \|A_1-A_2\|_{L_1} . \quad \qedhere
\end{multline*}
\end{proof}


\end{document}